\documentclass[12pt,english,reqno]{amsart}

\usepackage{amssymb}
\usepackage{amsmath}
\usepackage{amsthm}
\usepackage{enumitem}
\usepackage{cite}
\usepackage[usenames]{color}

\definecolor{webgreen}{rgb}{0,.5,0}
\definecolor{webbrown}{rgb}{.6,0,0}

\usepackage[
  pdfauthor={},
  pdfkeywords={},
  pdftitle={},
  pdfcreator={},
  pdfproducer={},
  colorlinks,linkcolor=blue,filecolor=webbrown,citecolor=webgreen]{hyperref}

\theoremstyle{plain}
\newtheorem{theorem}{Theorem}
\newtheorem{lemma}[theorem]{Lemma}
\newtheorem{corollary}[theorem]{Corollary}

\theoremstyle{definition}
\newtheorem{conjecture}[theorem]{Conjecture}

\theoremstyle{remark}
\newtheorem{remark}[theorem]{Remark}

\newcommand{\ZZ}{\mathbb{Z}}
\newcommand{\QQ}{\mathbb{Q}}

\newcommand{\DD}{\mathbb{D}}
\newcommand{\DB}{\mathfrak{D}}
\newcommand{\DN}{\mathbf{D}}
\newcommand{\DS}{\mathcal{D}}

\newcommand{\DDp}{\mathbb{D}^+}
\newcommand{\DDm}{\mathbb{D}^-}
\newcommand{\DDt}{\mathbb{D}^\top}
\newcommand{\DDb}{\mathbb{D}^\bot}
\newcommand{\DDc}{\mathbb{D}^{\top^{\scriptstyle\star}}}

\newcommand{\BN}{\mathbf{B}}
\newcommand{\MRS}{\mathcal{R}}
\newcommand{\MRo}{\overline{\mathcal{R}}}
\newcommand{\MS}{\mathcal{S}}
\newcommand{\MSo}{\overline{\mathcal{S}}}

\DeclareMathOperator{\denom}{denom}
\DeclareMathOperator{\lcm}{lcm}
\DeclareMathOperator{\rad}{rad}

\newcommand{\mids}{\,\mid\,}
\newcommand{\nmids}{\,\nmid\,}
\newcommand{\set}[1]{\{#1\}}
\newcommand{\intpart}[1]{\lfloor{#1}\rfloor}
\newcommand{\seqnum}[1]{\href{https://oeis.org/#1}{\rm \underline{#1}}}

\setlist[enumerate]{label=(\roman*),font=\rm,leftmargin=1.2cm,itemsep=1pt,
parsep=1pt,before={\parskip=1pt}}

\begin{document}

\title[On the finiteness of Bernoulli polynomials]
{On the finiteness of Bernoulli polynomials whose
derivative has only integral coefficients}
\author{Bernd C. Kellner}
\address{G\"ottingen, Germany}
\email{bk@bernoulli.org}
\subjclass[2020]{11B83 (Primary), 11B68 (Secondary)}
\keywords{Bernoulli polynomial, derivative, integral coefficient, denominator,
decomposition, product of primes, sum of base-$p$ digits}

\begin{abstract}
It is well known that the Bernoulli polynomials $\mathbf{B}_n(x)$ have nonintegral
coefficients for $n \geq 1$. However, ten cases are known so far in which the
derivative $\mathbf{B}'_n(x)$ has only integral coefficients.
One may assume that the number of those derivatives is finite.
We can link this conjecture to a recent conjecture about the properties of a
product of primes satisfying certain $p$-adic conditions.
Using a related result of Bordell{\`e}s, Luca, Moree, and Shparlinski,
we then show that the number of those derivatives is indeed finite.
Furthermore, we derive other characterizations of the primary conjecture.
Subsequently, we extend the results to higher derivatives of the Bernoulli
polynomials. This provides a product formula for these denominators,
and we show similar finiteness results.
\end{abstract}

\maketitle


\section{Introduction}

The Bernoulli polynomials $\BN_n(x)$ are defined by the exponential generating
function
\begin{equation} \label{eq:bpgf}
  \frac{t e^{xt}}{e^t - 1} = \sum_{n=0}^{\infty} \BN_n(x) \frac{t^n}{n!}
  \quad (|t| < 2\pi)
\end{equation}
and explicitly given by the formula
\begin{equation} \label{eq:bpdef}
  \BN_n(x) = \sum_{k=0}^{n} \binom{n}{k} \BN_{n-k} \, x^k \quad (n \geq 0),
\end{equation}
where ${\BN_n = \BN_n(0) \in \QQ}$ is the $n$th Bernoulli number.
It easily follows from \eqref{eq:bpgf} that \mbox{$\BN_n=0$} for odd ${n \geq 3}$.
For more properties see Cohen~\cite[Chapter~9]{Cohen:2007}.
The Bernoulli polynomials ${\BN_n(x) \in \QQ[x]}$ are
Appell polynomials \cite{Appell:1880}. Therefore, they satisfy the rule
\begin{equation} \label{eq:bpderiv}
  \BN'_n(x) = n \BN_{n-1}(x) \quad (n \geq 1).
\end{equation}
While ${\BN_n(x) \notin \ZZ[x]}$ for ${n \geq 1}$, which is equivalent to
${\denom(\BN_n(x)) > 1}$ for ${n \geq 1}$
(the denominators are discussed in the next section), it turns out that
\[
  \BN'_n(x) \in \ZZ[x] \quad \text{for }
  n \in \MS := \set{1, 2, 4, 6, 10, 12, 28, 30, 36, 60}.
\]

The elements of $\MS$, viewed as an ordered sequence, equal the finite
sequence~\seqnum{A094960} in the {\it On-Line Encyclopedia of Integer Sequences} 
(OEIS)~\cite{OEIS} as published in 2004. So far, no further terms have been found.
It is mainly assumed that sequence \seqnum{A094960} is indeed
finite and completely determined by $\MS$. Note that we implicitly omit 
the trivial case for $n=0$, since $\BN_0(x) = \BN_0 = 1$. Define
\[
  \MSo := \set{n \geq 1 : \BN'_n(x) \in \ZZ[x]}.
\]

For our purposes, we split the conjecture into two parts as follows.

\begin{conjecture} \label{conj:main}
We have the following statements.
\begin{enumerate}
\item The set $\MSo$ is finite.
\item We have $\MSo = \MS$.
\end{enumerate}
\end{conjecture}

We link the above conjecture to a more recent conjecture of the author~\cite{Kellner:2017}
in a $p$-adic context, where $p$ always denotes a prime.
The function $s_p(n)$ gives the sum of the base-$p$ digits of an integer $n \geq 0$.
Let $\omega(n)$ be the additive function that counts the distinct prime divisors of $n$.
As usual, an empty product is defined to be $1$, and an empty sum is defined to be $0$.
We consider the product
\begin{equation} \label{eq:ddp}
  \DDp_n := \prod_{\substack{p \, > \, \sqrt{n}\\ s_p(n) \, \geq \, p}} p \quad (n \geq 1),
\end{equation}
where $p$ runs over the primes. Note that the above product is always finite,
since $s_p(n) = n$ for $p > n$. By Kellner \cite[Theorem~4]{Kellner:2017},
we have a further relation that
\begin{equation} \label{eq:ddpo}
  \omega(\DDp_n) = \sum_{\substack{ p \, > \, \sqrt{n}\\
  \intpart{\frac{n-1}{p-1}} \, > \, \intpart{\frac{n}{p}}}} \!\! 1
  < \sqrt{n} \quad (n \geq 1).
\end{equation}

We shall clarify the notation of $\DDp_n$ in a more general setting in the next section.

\begin{conjecture}[Kellner~{\cite[Conjectures~1, 2]{Kellner:2017}}] \label{conj:kel}
We have the following statements.
\begin{enumerate}
\item We have $\DDp_n > 1$, respectively, $\omega(\DDp_n) > 0$ for $n > 192$.
\item There exists a constant $\kappa > 1$ such that
\begin{equation} \label{eq:ddpo2}
  \omega(\DDp_n) \, \sim \, \kappa \, \frac{\sqrt{n}}{\log n}
  \quad \text{as $n \to \infty$.}
\end{equation}
\end{enumerate}
\end{conjecture}

At first glance, Conjectures~\ref{conj:main} and \ref{conj:kel} seem to be incompatible.
However, we can establish the following connection.

\begin{theorem} \label{thm:main}
Conjecture~\ref{conj:kel}(i) and (ii) imply Conjecture~\ref{conj:main}(ii) and (i),
respectively.
\end{theorem}

Meanwhile, Conjecture~\ref{conj:kel}(ii) with $\kappa=2$ was proven
by Bordell{\`e}s et al.~\cite{BLMS:2018} for sufficiently large $n > n_0$.
This result was achieved by exploiting \eqref{eq:ddpo},
since the condition $s_p(n) \geq p$ as in~\eqref{eq:ddp} is replaced by
$\intpart{\frac{n-1}{p-1}} > \intpart{\frac{n}{p}}$ in~\eqref{eq:ddpo},
which enabled them to use powerful analytic tools.
Unfortunately, their methods do not lead to an explicit or computable bound $n_0$.
Using their results, we arrive at the following corollary.

\begin{corollary}[Bordell{\`e}s, Luca, Moree, and Shparlinski~{\cite[Corollary~1.6]{BLMS:2018}}]
Conjecture~\ref{conj:kel}(ii) is true, so Conjecture~\ref{conj:main}(i) is true.
\end{corollary}

\begin{theorem} \label{thm:main2}
We have the following statements.
\begin{enumerate}
\item If $n \in \MSo$, then $n+1$ is prime.
\item If $\MSo \neq \MS$ and $n \in \MSo \setminus \MS$, then $n > 10^7$.
\end{enumerate}
\end{theorem}

As a consequence, the set
\[
  \MSo + 1 = \set{2, 3, 5, 7, 11, 13, 29, 31, 37, 61, \ldots}
\]
contains only primes.
It would be very unlikely that $\DDp_n = 1$ happens for $n > 192$.
See the graph \cite[Figure~B1]{Kellner:2017} of $\omega(\DDp_n)$ in the range
below $10^7$ and consider the coincident and proven asymptotic formula of
$\omega(\DDp_n)$ for sufficiently large $n$.
However, it is still an open task to establish Conjectures~\ref{conj:main}(ii)
and~\ref{conj:kel}(i).

The paper is organized as follows. In the next section, we give a survey about
$p$-adic properties of the denominators of the Bernoulli polynomials.
We also show further characterizations of Conjecture~\ref{conj:main}.
In Section~\ref{sec:extend}, we extend the results to higher derivatives of the
Bernoulli polynomials. Section~\ref{sec:proof} contains the proofs of the theorems.


\section{Denominators and \texorpdfstring{$p$}{p}-adic properties}

To study the denominators of the Bernoulli polynomials,
it is convenient to consider for $n \geq 1$ the related denominators
\begin{align*}
  \DN_n &:= \denom(\BN_n) = 2,6,1,30,1,42,1,30,1,66,\ldots,\\
  \DD_n &:= \denom(\BN_n(x) - \BN_n) = 1,1,2,1,6,2,6,3,10,2,\ldots,\\
  \DB_n &:= \denom(\BN_n(x)) = 2,6,2,30,6,42,6,30,10,66,\ldots,
\end{align*}
which are all squarefree. These are the sequences \seqnum{A027642}, \seqnum{A195441},
and \seqnum{A144845}, respectively. Obviously, we have by definition the relation
\begin{equation} \label{eq:dbdn}
  \DB_n = \lcm(\DD_n, \DN_n).
\end{equation}
The denominators $\DN_n$ of the Bernoulli numbers are given by the well-known
von Staudt--Clausen theorem of $1840$
(Clausen~\cite{Clausen:1840} and von Staudt~\cite{Staudt:1840}),
which states for even positive integers~$n$ that
\[
  \BN_n + \sum_{p-1 \mids n} \frac{1}{p} \in \ZZ,
  \quad \text{which implies that} \quad
  \DN_n = \prod_{p-1 \mids n} p.
\]
However, the denominators $\DN_n$ do not play a role here, since we follow the
approaches of Kellner~\cite{Kellner:2017} and Kellner and Sondow
\cite{Kellner&Sondow:2017,Kellner&Sondow:2018,Kellner&Sondow:2021},
which are concerned with the $p$-adic properties of the denominators $\DD_n$.
For $n \geq 1$, these denominators are given by the remarkable formula
\begin{equation} \label{eq:ddprod}
  \DD_n = \prod_{s_p(n) \, \geq \, p} p,
\end{equation}
which arises from the $p$-adic product formula;
see Kellner~\cite[Section~5]{Kellner:2017}. The decomposition
\begin{equation} \label{eq:decomp}
  \DD_n = \DDm_n \,\cdot\, \DDp_n,
\end{equation}
where $\DDp_n$ is defined as in \eqref{eq:ddp} and
\[
  \DDm_n := \prod_{\substack{p \,<\, \sqrt{n}\\ s_p(n) \,\geq\, p}} p,
\]
leads to Conjecture~\ref{conj:kel}. Note that the decomposition~\eqref{eq:decomp}
omits the possible term for $p = \sqrt{n}$, but then we would have $p^2 = n$ and
so $s_p(n) = 1$.

For computational purposes, those products, which run over the primes $p$ and
contain the condition $s_p(n) \geq p$, are trivially bounded by $p < n$.
Moreover, the following bounds \cite[Lemmas~1, 2]{Kellner:2017} are self-induced
by properties of $s_p(n)$. Namely, we have for $n \geq 1$ that
\[
  s_p(n) < p, \quad \text{if\ } p > \frac{n+1}{\lambda} \text{\ where\ }
  \lambda = \begin{cases}
    2, & \text{if $n$ is odd;} \\
    3, & \text{if $n$ is even.}
  \end{cases}
\]

For the sake of completeness, we show that the polynomials $\BN_n(x) - \BN_n$,
which have no constant term, arise in a natural context. For $n \geq 0$,
define the sum-of-powers function
\[
  S_n(m) := \sum_{\nu=0}^{m-1} \nu^n \quad (m \geq 0).
\]
It is well known that
\begin{equation} \label{eq:powint}
  S_n(x) = \int_{0}^{x} \, \BN_n(t) \, dt = \frac{1}{n+1}(\BN_{n+1}(x) - \BN_{n+1}).
\end{equation}
As a result of Kellner \cite[Theorem~5]{Kellner:2017} and Kellner and Sondow
\cite[Theorems~1, 2]{Kellner&Sondow:2017}, we then have for $n \geq 0$ that
\begin{equation} \label{eq:ds}
  \DS_n := \denom(S_n(x)) = (n+1) \, \DD_{n+1}
  = 1, 2, 6, 4, 30, 12, 42, 24, 90, 20, \ldots,
\end{equation}
which is sequence \seqnum{A064538}.

\begin{remark}
Since the Bernoulli polynomials $\BN_n(x)$ are Appell polynomials satisfying the
reflection relation $\BN_n(1-x) = (-1)^n \, \BN_n(x)$, the integral in \eqref{eq:powint}
can be reinterpreted by Faulhaber-type polynomials that are connected with certain
reciprocal Bernoulli polynomials, as recently shown by the author; see
\cite[Example~5.6]{Kellner:2021} and \cite[Section~11]{Kellner:2023}.
\end{remark}

Let ${n \geq 1}$, and let $\rad(n)$ be the squarefree kernel of $n$.
As introduced in \cite{Kellner&Sondow:2021}, define the decompositions
\begin{equation} \label{eq:decomp2}
  \DD_n = \DDt_n \,\cdot\, \DDb_n \quad \text{and} \quad
  \rad(n) = \DDt_n \,\cdot\, \DDc_n,
\end{equation}
where
\begin{equation} \label{eq:ddprod2}
  \DDt_n := \prod_{\substack{p \mids n\\ s_p(n) \, \geq \, p}} p, \quad
  \DDb_n := \prod_{\substack{p \nmids n\\ s_p(n) \, \geq \, p}} p, \quad \text{and} \quad
  \DDc_n := \prod_{\substack{p \mids n\\ s_p(n) \, < \, p}} p.
\end{equation}
The sequences of $\DDt_n$, $\DDb_n$, and $\DDc_n$
are sequences \seqnum{A324369}, \seqnum{A324370}, and \seqnum{A324371}, respectively.
We arrive at the following theorem.

\begin{theorem}[Kellner and Sondow~{\cite[Theorem~3.1]{Kellner&Sondow:2021}}]
\label{thm:triple}
For $n \geq 1$, the denominator $\DB_n$ of the Bernoulli polynomial $\BN_n(x)$
splits into the triple product
\[
  \DB_n = \DDb_{n+1} \,\cdot\, \DDt_{n+1} \,\cdot\, \DDc_{n+1}.
\]
\end{theorem}

Consequently, the interplay of the factors of $\DB_n$ yields the relations
\begin{equation} \label{eq:dbrel}
  \DB_n = \DDb_{n+1} \,\cdot\, \rad(n+1) = \DD_{n+1} \,\cdot\, \DDc_{n+1}
  = \lcm(\DD_{n+1}, \rad(n+1)).
\end{equation}
Compared with the classical relation \eqref{eq:dbdn},
one may observe that the right-hand sides of the above equations
involve the numbers and indices $n+1$ in place of $n$.
To simplify notation, we include the case $\DB_0 = 1$,
which coincides with \eqref{eq:dbrel}. Apart from that, we explicitly avoid
the case ${n=0}$ of the related symbols of $\DD_n$ in view of
their product identities \eqref{eq:decomp2} and \eqref{eq:ddprod2}.

\begin{corollary}
Let $n \geq 1$. The following statements hold.
\begin{enumerate}
\item We have that $\DB_n = \rad(\DS_n)$.
\item We have that $\DB_n$ is even, which implies that $\BN_n(x) \notin \ZZ[x]$.
\item We have that $\DDb_n$ is even, if $n \geq 3$ is odd; otherwise, $\DDb_n$ is even.
\end{enumerate}
\end{corollary}

A different proof of part (ii) via \eqref{eq:dbdn} is given by Kellner and Sondow
\cite[Theorem~4]{Kellner&Sondow:2017}.

\begin{proof}
Let $n \geq 1$. We show three parts.
(i).~From \eqref{eq:ds} and \eqref{eq:dbrel}, we derive that
$\DB_n = \lcm(\DD_{n+1}, \rad(n+1)) = \rad(\DD_{n+1}\,(n+1)) = \rad(\DS_n)$.
(iii).~We have $\DDb_1 = 1$. If $2 \mid n$, then $2 \nmid \DDb_n$.
Otherwise, for odd $n \geq 3$, it follows that $2 \nmid n$ and $s_2(n) \geq 2$.
By~\eqref{eq:ddprod2}, this shows that $2 \mid \DDb_n$.
(ii).~Considering \eqref{eq:dbrel}, the factor $\rad(n+1)$ is even for odd $n$,
whereas $2 \mid \DDb_{n+1}$ when $n$ is even using part (iii).
Both cases show that $\DB_n$ is even for $n \geq 1$. This completes the proof.
\end{proof}

The properties of $\DD_n$ and $\DDb_n$ lead to the following characterizations,
which are connected with Conjecture~\ref{conj:main}.
For this purpose, we define the sets
\[
  \MRo := \set{n \geq 1 : \DD_n = \rad(n+1)} \quad \text{and} \quad
  \MRS := \set{3,5,8,9,11,27,29,35,59}.
\]

\begin{theorem} \label{thm:deriv}
Let $n \geq 1$. We have that $\BN'_n(x) \in \ZZ[x]$ if and only if $\DDb_n = 1$,
or equivalently, $\DB_{n-1} = \rad(n)$. In these cases, the number $n+1$ is prime.
\end{theorem}

\begin{theorem} \label{thm:rad}
If $n \in \MRo$,
then $n+1$ is composite. In particular, if $n$ is odd, then $\DDb_{n+1} = 1$.
Otherwise, we have that $n = 2^e$ for some $e \geq 1$.
Moreover, the set $\MRo$ is finite.
\end{theorem}

\begin{conjecture} \label{conj:rad}
We have $\MRo = \MRS$.
\end{conjecture}

\begin{theorem} \label{thm:rad2}
Conjecture~\ref{conj:rad}, reduced to odd numbers $n \in \MRS$,
implies Conjecture~\ref{conj:main}.
\end{theorem}


\section{Denominators and higher derivatives}
\label{sec:extend}

In this section, we extend the results to higher derivatives of $\BN_n(x)$.
Let $(n)_k$ denote the falling factorial such that $\binom{n}{k} = (n)_k / k!$.
We define the related denominators by
\[
  \DB_n^{(k)} := \denom(\BN_n^{(k)}(x)) \quad (k, n \geq 1).
\]

\begin{theorem} \label{thm:high}
Let $k, n \geq 1$. Then we have
\begin{equation} \label{eq:dbderiv}
  \DB_n^{(k)} = \frac{\DB_{n-k}}{\gcd(\DB_{n-k},(n)_k)}
  = \frac{\DDb_{n-k+1}}{\gcd(\DDb_{n-k+1},(n)_{k-1})}
  = \prod_{\substack{p \nmids (n)_k\\ s_p(n-k+1) \, \geq \, p}} p \quad (n \geq k).
\end{equation}
Otherwise, we have $\DB_n^{(k)} = 1$.
Moreover, the denominators $\DB_n^{(k)}$ have the property that
$p \nmid \DB_n^{(k)}$ for all primes $p \leq k$. In particular, we have
\[
  \DB_n^{(1)} = \frac{\DB_{n-1}}{\rad(n)} = \DDb_n \quad (n \geq 1).
\]
\end{theorem}

Define the related sets
\[
  \MSo_k := \set{n \geq 1 : \BN^{(k)}_n(x) \in \ZZ[x]} \quad (k \geq 1),
\]
where $\MSo_1 = \MSo$.
Let $\MS_k$ denote the computable subsets of $\MSo_k$ with $\MS_1 = \MS$.

\begin{theorem} \label{thm:high2}
We have that all sets $\MSo_k$ are finite for $k \geq 1$.
Moreover, we have that
\[
  \MSo_1 \subset \MSo_2 \subset \MSo_3 \subset \cdots.
\]
\end{theorem}

Recall that
\[
  \MS_1 = \set{1, 2, 4, 6, 10, 12, 28, 30, 36, 60}.
\]
We use the notation, e.g., $\set{a-b} = \set{a,\ldots,b}$ for ranges.
We compute the sets
\begin{align*}
  \MS_2 = \{ &1-7,9-13,15,16,21,25,28-31,36,37,55,57,60,61,70,121,190 \}, \\
  \MS_3 = \{ &1-18, 20-22, 25, 26, 28-32, 35-38, 42, 50, 52, 55-58, 60-62, \\
  & 66, 70-72, 78, 80, 92, 110, 121, 122, 156, 176, 177, 190, 191, 210, 392 \}.
\end{align*}

\begin{conjecture}
We have $\MSo_2 = \MS_2$ and $\MSo_3 = \MS_3$.
\end{conjecture}


\section{Proofs of the theorems}
\label{sec:proof}

We give the proofs of the theorems in the order of their dependencies.
First, we need some key lemmas.

\begin{lemma} \label{lem:ddprop}
Let $n \geq 1$. We have the following properties.
\begin{enumerate}
\item We have that $\DD_n$ is odd if and only if $n = 2^e$ for some $e \geq 0$.
\item If $n+1$ is composite, then $\rad(n+1) \mid \DD_n$ and $\rad(n+1) \mid \DDb_n$.
\item If $n \geq 3$ is odd, then $\DD_n = \lcm(\DD_{n+1}, \rad(n+1))$.
\end{enumerate}
\end{lemma}

\begin{proof}
(i).~See \cite[Theorem~1]{Kellner&Sondow:2017}.
(ii).~See \cite[Theorem~1]{Kellner:2017} and \cite[Corollary~2]{Kellner&Sondow:2018},
respectively.
(iii).~See \cite[Theorem~3.2]{Kellner&Sondow:2021}.
\end{proof}

\begin{lemma} \label{lem:dddiv}
For $n \geq 1$, we have $\DDp_n \mid \DDb_n$.
\end{lemma}

\begin{proof}
Let $n \geq 1$, and assume that $\DDp_n > 1$. Otherwise, we are done.
If a prime $p \mid \DDp_n$, then by \eqref{eq:ddp} we have
$p > \sqrt{n}$ and $s_p(n) \geq p$. Thus, we have $p^2 > n$ implying
the $p$-adic expansion $n = a_0 + a_1 p$.
It then follows from $a_0 + a_1 = s_p(n) \geq p$ that $a_0 \neq 0$
and $p \nmid n$. Finally, this shows by \eqref{eq:ddprod2} that $p \mid \DDb_n$.
\end{proof}

\begin{lemma} \label{lem:ddestim}
For $k \geq 1$, there exists a number $n_k$ such that
$\DDp_n > (n+k)_k$ for every $n > n_k$.
\end{lemma}

\begin{proof}
Let $k \geq 1$. Using the result of Bordell{\`e}s et al.~\cite[Corollary~1.6]{BLMS:2018},
we have that \eqref{eq:ddpo2} holds with $\kappa = 2$.
Thus, there exists a number $n_k$ such that $\omega(\DDp_n) \geq 2k$ for $n > n_k$.
Now consider any set of $2k$ prime divisors $p_1 < p_2 < \cdots < p_{2k}$ of $\DDp_n$.
Since $p \mid \DDp_n$ implies $p > \sqrt{n}$, we infer that
$n+1 < p_1 p_2 < p_3 p_4 < \cdots < p_{2k-1} p_{2k}$.
Consequently, we get $\DDp_n > (n+1)\cdots(n+k) = (n+k)_k$ for every $n > n_k$.
\end{proof}

\begin{proof}[Proof of Theorem~\ref{thm:deriv}]
Let $n \geq 1$. If $\BN'_n(x) \in \ZZ[x]$, then we have by \eqref{eq:bpderiv} that
\[
  \denom(\BN'_n(x)) = \denom(n \, \BN_{n-1}(x)) = 1.
\]
As a consequence, we have that $\DB_{n-1} \mid n$. Applying Theorem~\ref{thm:triple}
and \eqref{eq:dbrel}, it follows that
\begin{equation} \label{eq:ddcond}
  \DDb_n \,\cdot\, \rad(n) \mid n.
\end{equation}
Since $\DDb_n$ is coprime to $n$, we conclude that $\DDb_n = 1$.
In the other direction, condition~\eqref{eq:ddcond} is satisfied
only if $\DDb_n = 1$. By \eqref{eq:dbrel}, this is equivalent to
$\DB_{n-1} = \rad(n)$. From Lemma~\ref{lem:ddprop}(ii), it further follows that
$n+1$ must be prime. Otherwise, we would have that $\rad(n+1) \mid \DDb_n$.
This completes the proof.
\end{proof}

\begin{proof}[Proof of Theorem~\ref{thm:main}]
First, we assume Conjecture~\ref{conj:kel}(i). Let $n > 192$. We then have that
$\DDp_n > 1$. By Lemma~\ref{lem:dddiv}, this property transfers to $\DDb_n > 1$.
From Theorem~\ref{thm:deriv}, it follows that $\BN'_n(x) \notin \ZZ[x]$ for $n > 192$.
We have to check the remaining cases $1 \leq n \leq 192$.
By Theorem~\ref{thm:deriv}, it then suffices to check the numbers $\DDb_n$
when $n+1$ is prime. Finally, this confirms that $\BN'_n(x) \in \ZZ[x]$
if and only if $n \in \MS$, implying Conjecture~\ref{conj:main}(ii).
Secondly, we now assume Conjecture~\ref{conj:kel}(ii). It follows from \eqref{eq:ddpo2}
that there exists a number $n_0$ such that $\DDp_n > 1$ for $n > n_0$.
Similar to the first part above, this implies that
$\BN'_n(x) \notin \ZZ[x]$ for $n > n_0$.
Hence, the set $\MSo$ is finite, which is Conjecture~\ref{conj:main}(i).
\end{proof}

\begin{proof}[Proof of Theorem~\ref{thm:main2}]
We have to show two parts.
(i).~By Theorem~\ref{thm:deriv}, we have that $n \in \MSo$ implies that $n+1$ is prime.
(ii).~This follows from computations of the graph \cite[Figure~B1]{Kellner:2017}
of $\omega(\DDp_n)$ in the range below $10^7$.
\end{proof}

\begin{proof}[Proof of Theorem~\ref{thm:rad}]
Let $n \in \MRo$. Assume that $\DD_n = \rad(n+1)$, where $p = n+1$ is prime.
Then $\DD_n = p$ contradicts \eqref{eq:ddprod}, since $s_p(n) = n < p$.
Therefore, the number $n+1$ is composite.
If $n$ is even, then $\DD_n = \rad(n+1)$ implies that $\DD_n$ is odd.
From Lemma~\ref{lem:ddprop}(i), it follows that $n = 2^e$ for some $e \geq 1$.
Now, we assume that $n \geq 3$ is odd, neglecting the case $\DD_1=1$.
Using Lemma~\ref{lem:ddprop}(iii), we infer that
\[
  \DD_n = \lcm(\DD_{n+1}, \rad(n+1))
  = \DDb_{n+1} \lcm(\DDt_{n+1}, \rad(n+1)) = \rad(n+1),
\]
so $\DDb_{n+1} = 1$ as desired. It remains to show that $\MRo$ is finite.
Applying Lemma~\ref{lem:ddestim} with $k=1$ shows that there exists $n_1$
such that $\DDp_n > n+1$ for $n > n_1$. Since $\DD_n \geq \DDp_n$ by \eqref{eq:decomp},
it follows that $\MRo$ is finite. This completes the proof.
\end{proof}

\begin{proof}[Proof of Theorem~\ref{thm:rad2}]
Let $\MRS' = \set{3,5,9,11,27,29,35,59}$ be the reduced set of $\MRS$ consisting
only of odd numbers. From Theorems~\ref{thm:deriv} and~\ref{thm:rad}, we derive
that $n \in \MRS'$ implies that $\DDb_{n+1} = 1$ and so $n+1 \in \MS$.
Since $\MRS'+1 = \MS \setminus \set{1,2}$, the result follows.
\end{proof}

\begin{proof}[Proof of Theorem~\ref{thm:high}]
First assume that $1 \leq n < k$. By \eqref{eq:bpdef}, the Bernoulli polynomial
$\BN_n(x)$ is a monic polynomial of degree $n$.
Thus, the $k$th derivative of $\BN_n(x)$ vanishes, yielding $\DB_n^{(k)} = 1$.
If $n = k \geq 1$, then $\BN^{(k)}_n(x) = n!$ and so $\DB_n^{(k)} = 1$.
As $\DB_0 = \DDb_1 = 1$, this case coincides with \eqref{eq:dbderiv}.
Now, let $n > k \geq 1$. From \eqref{eq:bpderiv}, it follows that
\[
  \DB_n^{(k)} = \denom(\BN^{(k)}_n(x)) = \denom((n)_k \, \BN_{n-k}(x))
  = \frac{\DB_{n-k}}{\gcd(\DB_{n-k},(n)_k)}.
\]
Using \eqref{eq:dbrel}, we have $\DB_{n-k} = \DDb_{n-k+1} \rad(n-k+1)$.
Recall that $\DB_{n-k}$ is squarefree.
Together with $(n)_k = (n)_{k-1} (n-k+1)$, we infer that
\[
  \gcd(\DB_{n-k},(n)_k) = \gcd(\DDb_{n-k+1},(n)_{k-1}) \rad(n-k+1).
\]
As a consequence, we obtain
\begin{equation} \label{eq:dbderiv2}
  \DB_n^{(k)} = \frac{\DDb_{n-k+1}}{\gcd(\DDb_{n-k+1},(n)_{k-1})}.
\end{equation}
Note that $p \mid (n)_{k-1}$ implies $p \nmid \DB_n^{(k)}$.
By \eqref{eq:ddprod2}, we have
\begin{equation} \label{eq:ddprod3}
  \DDb_{n-k+1} = \prod_{\substack{p \nmids n-k+1\\ s_p(n-k+1) \, \geq \, p}} p.
\end{equation}
Putting \eqref{eq:dbderiv2} and \eqref{eq:ddprod3} together, we get
\[
  \DB_n^{(k)} = \prod_{\substack{p \nmids (n)_k\\ s_p(n-k+1) \, \geq \, p}} p.
\]
From $k! \mid (n)_k$ and $p \nmid (n)_k$, we derive that
$p \nmid \DB_n^{(k)}$ for $p \leq k$.
At the end, we consider the case $k=1$ for $n \geq 1$.
Recall that $\DB_1^{(1)} = \DDb_1 = 1$.
For $n > 1$, we have $\DB_n^{(1)} = \DDb_n$ by \eqref{eq:dbderiv2}.
From \eqref{eq:dbrel} and $\DB_0 = 1$,
it finally follows that $\DB_n^{(1)} = \DB_{n-1} / \rad(n) = \DDb_n$
for all $n \geq 1$.
\end{proof}

\begin{proof}[Proof of Theorem~\ref{thm:high2}]
Let $k \geq 1$. Combining Lemmas~\ref{lem:dddiv} and \ref{lem:ddestim},
we see that there exists a number $n_k$ such that
\[
  \DDb_n \geq \DDp_n > (n+k)_k \quad (n > n_k).
\]
By Theorem~\ref{thm:high} and shifting the index $n$ to $n+k-1$ in \eqref{eq:dbderiv},
we obtain
\[
  \DB_{n+k-1}^{(k)} = \frac{\DDb_{n}}{\gcd(\DDb_{n},(n+k-1)_{k-1})} \quad (n \geq 1).
\]
Since $\gcd(\DDb_{n},(n+k-1)_{k-1}) \leq (n+k-1)_{k-1} < (n+k)_k$, we then deduce that
\[
  \DB_{n+k-1}^{(k)} > \frac{\DDb_{n}}{(n+k)_k} > 1 \quad (n > n_k),
\]
showing that $\MSo_k$ is finite.
As $\BN^{(k)}_n(x) \in \ZZ[x]$ also implies that $\BN^{(k+1)}_n(x) \in \ZZ[x]$,
we infer that $\MSo_k \subset \MSo_{k+1}$.
Hence, this yields $\MSo_1 \subset \MSo_2 \subset \MSo_3 \subset \cdots$,
completing the proof.
\end{proof}

\section{Acknowledgments}

The author would like to thank the anonymous referee for the careful reading
and for several valuable comments that improved the quality of the paper.


\bibliographystyle{amsplain}

\end{document}